\documentclass[a4paper,12pt,twoside]{article}

\usepackage{a4 wide}

\usepackage[T1, T2A]{fontenc}
\usepackage[utf8]{inputenc}
\usepackage[english]{babel}

\usepackage{amsmath}
\usepackage{amssymb}
\usepackage{amsthm}
\usepackage{amscd}
\usepackage{cite}
\usepackage[matrix,arrow,curve]{xy}

\newtheorem{theorem}{Theorem}[section]
\newtheorem{lemma}[theorem]{Lemma}
\newtheorem{proposition}[theorem]{Proposition}
\newtheorem{corollary}[theorem]{Corollary}

\newtheorem{definition}[theorem]{Definition}

\newcommand{\Sob}{\!\!\vbox{\hbox{\,\,\,\,\tiny $\circ$}\vspace*{-1.0ex} \hbox{ $H^{1}$}\vspace*{-0.0ex}}}
\newcommand{\Wob}{\!\!\vbox{\hbox{\,\,\,\,\tiny $\circ$}\vspace*{-1.0ex} \hbox{ $W^{1}_p$}\vspace*{-0.8ex}}}
\newcommand{\Wobm}{\!\!\vbox{\hbox{\,\,\,\,\tiny $\circ$}\vspace*{-1.0ex} \hbox{ $W^{m}_p$}\vspace*{-0.8ex}}}
\newcommand{\SobSm}{\!\!\vbox{\hbox{\,\,\,\tiny $\circ$}\vspace*{-2.0ex} \hbox{  \scriptsize $H^{1}$}\vspace*{-0.0ex}}}
\newcommand{\WobSm}{\!\!\vbox{\hbox{\,\,\,\tiny $\circ$}\vspace*{-2.0ex} \hbox{  \scriptsize $W^{1}_p$}\vspace*{-0.6ex}}}
\newcommand{\WobSmd}{\!\!\vbox{\hbox{\,\,\,\tiny $\circ$}\vspace*{-2.0ex} \hbox{  \scriptsize $W^{2}_p$}\vspace*{-0.6ex}}}
\newcommand{\WobSmi}{\!\!\vbox{\hbox{\,\,\,\tiny $\circ$}\vspace*{-2.0ex} \hbox{  \scriptsize $W^{1}_\infty$}\vspace*{-0.4ex}}}
\newcommand{\SobSl}{\!\!\vbox{\hbox{\,\,\,\tiny $\circ$}\vspace*{-1.6ex} \hbox{  \scriptsize $H^{1}$}\vspace*{-0.0ex}}}
\newcommand{\WobSl}{\!\!\vbox{\hbox{\,\,\,\tiny $\circ$}\vspace*{-1.6ex} \hbox{  \scriptsize $W^{1}_p$}\vspace*{-0.6ex}}}
\newcommand{\Sobm}{\!\!\vbox{\hbox{\,\,\,\,\tiny $\circ$}\vspace*{-1.0ex} \hbox{ $H^{m}$}\vspace*{-0.0ex}}}
\newcommand{\Sobd}{\!\!\vbox{\hbox{\,\,\,\,\tiny $\circ$}\vspace*{-1.0ex} \hbox{ $H^{2}$}\vspace*{-0.0ex}}}

\begin{document}

\begin{center}
\textbf{\large\bf{\sc On the smoothness of solutions to elliptic equations in domains with Hölder boundary.}}
\end{center}
\begin{center}
 \textbf{I.\,V.~Tsylin}
\end{center}

\textbf{Abstract:} The dependence of the smoothness of variational solutions to the first boundary value problems for second order elliptic operators are studied. The results use Sobolev-Slobodetskii and Nikolskii-Besov spaces and their properties. Methods are based on real interpolation technique and generalization of Savar\'{e}-Nirenberg difference quotient technique.

\section{Introduction}

Let $(M, g)$ be a smooth connected compact oriented Riemannian manifold without boundary, $\Omega \subsetneq M$ be a subdomain with a Hölder boundary. The aim of this paper is to study the dependence of the smothness of the variational solutions to the following equation:
\begin{equation}
\label{prob}
\mathcal{A}u = f, \,\,\,\, u \in \Sob(\Omega),
\end{equation}
on the regularity of the right hand-side $f \in H^{-1 + \varepsilon}(M)$, $\varepsilon > 0$.  By definition, the operator $\mathcal{A}$ is generated by the continuous positive bilinear form $\Phi$ defined on $\Sob(\Omega)$, associated with the differential operation $\mathcal{A}'$, which is locally represented as follows:
\begin{equation}
\label{gen_eq}
-\frac{1}{\sqrt{\det g}} \partial_i \left( \sqrt{\det g}\,\, a^{ij} \partial_j u \right) + b^i \partial_i u + c u;
\end{equation}
where $a^{ij}$, $b^{i}$, $c$ are sufficiently regular coefficients.

By ellipticity, if the right hand-side belongs to $L_2(\Omega)$, then the solution of (1) belongs to $H^2_{loc}(\Omega)$. One cannot replace $H^2_{loc}(\Omega)$ with $H^2(\Omega)$ even if the boundary $\partial\Omega$ is Lipschitz continuous \cite{Grisvard}. However, this is possible whenever $\Omega$ is a convex set or $\partial \Omega \in C^{1,1}$ (\cite{Grisvard}, Theorems 2.2.2.3, 3.2.1.2).

Suppose $\Omega \subset \mathbb{R}^d$ is a bounded domain with a Lipschitz boundary, $\mathcal{A} = -\Delta$, and $\tilde{H}^{-1+s}(\Omega)$ is the space of all functions $v\in H^{1+s}(M)$, such that $\mathrm{supp} \,v \subset \bar{\Omega}$. It was shown by Jerison and Kenig (in \cite{Kenig}) that if $f \in H^{-1+s}(\Omega)$, $s \in [0,1/2)$, then the solution $u \in \tilde{H}^{1+s}(\Omega)$. In \cite{Savare98}, G. Savar\'{e} elaborated a new method to generalize this Proposition to the case of Lipschitz coefficients.

\textbf{Theorem\,([11]).}
\textit{Let $\Omega \subset \mathbb{R}^d$ be a bounded domain, $\partial \Omega$ be a Lipschitz continuous boundary, $\mathcal{A}$ be generated by (\ref{gen_eq}), and  $a^{ij} \in C^{0,1}(\bar{\Omega})$ be a symmetric positive definite matrix in $\bar{\Omega}$, $b^i \equiv 0$, $c \equiv 0$. Then, the solution of (\ref{prob}) belongs to $\tilde{H}^{1+s}(\Omega)$, $s \in [0,1/2)$, whenever $f\in H^{-1+s}(\Omega)$.}

In this paper we establish similar results in the situation when both the boundary and the coefficients are Hölder continuous. One of the results is the following (the proof will be given in Section \ref{Results}).

\begin{theorem}
\label{result}
Let $M$ be a $d$--dimensional  $C^{1,1}$--smooth compact Riemannian manifold without boundary, a domain $\Omega \subsetneq M$ be such that $\partial \Omega$ is Hölder continuous of order $\gamma_\Omega in (0,1]$. Moreover, let $\mathcal{A}$ be generated by (\ref{gen_eq}), and for some $\varepsilon > 0$ the coefficients $a^{ij}$ and $b^i$ define a symmetric positive definite $C^{0,\gamma_c}(M)$-smooth section of $T^2 M$, $L_{\frac{d+\varepsilon}{1-\gamma_c}}(\Omega)$-section $TM$ respectively, $c \in W^{-1 + \gamma_c + \epsilon}_d(\Omega)$ with $0 < \gamma_c \leq 1$, ($c\in L_{\max\{d, 2+ \epsilon\}}(\Omega)$ if $\gamma_c = 1$) and the form $\Phi$ be positive in $\Sob(\Omega)$. Then the operator 
$$
\mathcal{R}: H^{-1+s}(M) \rightarrow \tilde{H}^{1 + \gamma_\Omega s}(\Omega), \,\,\,\, s\in [0,\gamma_c /2).
$$
solving problem (\ref{prob}) is continuous.
\end{theorem}

Our method is based on ideas in \cite{Savare02} and \cite{StepinTsylin}.

\section{Terms and сoncepts}
\subsection{Domain $\Omega$}
\label{Boundary}

Further assume that $(M,g)$ is a $C^{1,1}$--smooth connected oriented compact manifold without boundary and that every coordinate mapping acts to $\mathbb{R}^d$ with fixed Euclidean norm $|\cdot|$.

\begin{definition}
\label{boundary}
A non-empty open set $\Omega \subset M$ is called a domain with a Hölder continuous boundary of order $\gamma_\Omega$ ($\partial \Omega \in C^{0,\gamma_\Omega}$) if there is an atlas\footnote{Here $V$ is an open subset of $M$ and $\kappa_V: V \to \tilde{V} \subset \mathbb{R}$ is a diffeomorphism} $\mathcal{V} = \{ (V, \kappa_V)\}$ such that for any $(V, \kappa_V) \in\mathcal{V}$ there eixst a unit vecot $\xi_V \in \mathbb{R}^d$ and a function $g_V: \xi_V^{\bot} \to \mathbb{R}$ such that $g_V \in C^{0,\gamma_\Omega}(\xi_V^{\bot})$ with $0 < \gamma_\Omega \leq 1$, $\kappa_V(V\cap \partial \Omega)$ is a subset of the graph of $g_v$ and the intersection of $\kappa_V (V \cap \Omega)$ with the epigraph of $g_V$ is empty.
\end{definition}

\subsection{Operator $\mathcal{A}$}

Let us suppose that $a^{ij}$ and $b^i$ in (\ref{gen_eq}) define a symmetric positive definite $C^{0,\gamma_c}(M)$-smooth section $\textbf{A}$ of $T^2 M$, and $L_{\frac{d}{1-\gamma_c}}(\Omega)$-section $\textbf{b}$ of $TM$ respectively. We shall denote by $\textbf{G}$ the section of $T^2 M$ generated by the Riemannian structure $g$. Since $\textbf{A}$ and $\textbf{G}$ are dependent on $x\in M$, we denote them as $\textbf{A}_x$ and $\textbf{G}_x$. The following conditions are assumed

\begin{enumerate}
\item[\textbf{A1}] There exists a constant $\alpha > 0$ such that
$$
\forall x \in M\,\, \forall \xi \in T^*_x M \Rightarrow \alpha \textbf{G}_x(\xi, \xi) \leq \textbf{A}_x (\xi, \xi);
$$
\item[\textbf{A2}] Section $\textbf{A}$ belongs to $C^{0,\gamma_c}(M)$, $\gamma_c \in (0,1]$.
\end{enumerate}
We endow $C^{0,\gamma_c}(M)$ by the following norm:
$$
\| \textbf{A} \|_{C^{0,\gamma_c}(M)} \stackrel{\mathrm{def}}{=} \| \textbf{A} \|_{C(M)} + [\textbf{A}]_{C^{0,\gamma_c}(M)},
$$
where $\| \textbf{A} \|_{C(M)} = \max_{x \in M} \max_{\xi \in T^*_x M, \xi \neq 0} \frac{\textbf{A}_x(\xi,\xi)}{\textbf{G}_x(\xi,\xi)}$, $[\textbf{A}]_{C^{0,\gamma_c}(M)} = \sum_{U} \max_{ij} \left[ a^{ij}_U \right]_{C^{0,\gamma_c}(U)}$,
$$
[v]_{C^{0,\gamma_c}(U)} = \sup_{x,y \in U, x\neq y} \frac{|v(x) - v(y)|}{|x-y|}, \,\,\,\, v: U \to \mathbb{R},
$$
Here $\mathcal{U} = \{(U, \kappa_U)\}$ is a fixed finite atlas for $M$, and $a^{ij}_U$ are coordinates of $\textbf{A}$ with respect to the maps $\kappa_U$. It can easily be proved that the convergence in $C^{0,\gamma_c}(M)$ is independent of $\mathcal{U} = \{(U, \kappa_U)\}$.

Suppose that
$$
(u,v)_{\SobSl(\Omega)} = \int_{\Omega} \textbf{G}(\nabla u, \nabla v) d\mu, \,\,\,\, (u,v)_{L_2(\Omega)} = \int_\Omega uv d\mu,
$$
are inner products in $\Sob(\Omega)$, $L_2(\Omega)$ respectively, and that the measure $\mu$ is associated with the  Riemannian structure $g$. Let for $1\leq p \leq \infty$
\begin{align*}
&\| u \|_{L_p(\Omega)} = \left(\int_{\Omega} |u|^p d\mu \right)^{1/p},\,\,\,\,\, &\| u \|_{\WobSm(\Omega)} = \left( \int_\Omega \left| \textbf{G}(\nabla u, \nabla u) \right|^{p/2} d\mu \right)^{1/p},\,\,\,\,\,\, &p \in [1,\infty) \\
&\| u \|_{L_\infty(\Omega)} = \mathrm{ess}\sup_{x\in \Omega} |u(x)|,\,\,\,\,\, &\| u \|_{\WobSmi(\Omega)} = \mathrm{ess}\sup_{x\in \Omega} \left| \textbf{G}_x (\nabla u, \nabla u) \right|^{1/2}. \,\,\,\,\,\,&
\end{align*}
Further denote
$$
\| v \|_{W^{-1}_q(\Omega)} = \sup_{u \in \WobSl(\Omega), u \neq 0} \frac{|v(u)|}{\| u \|_{\WobSl(\Omega)}}, p \in (1,\infty), 1/p + 1/q = 1.
$$

\begin{enumerate}
\item[\textbf{A3}] Assume that  $b^i$, $c$ are such that the form $\Phi$ is positive and continuous in $\Sob(\Omega)$.
\end{enumerate}

Consider the following representation of $\Phi$:
\begin{align*}
\Phi (u,v) &= \Phi_0 (u,v) + \Phi_r(u,v),\\
\Phi_0 (u,v) = \int_M \textbf{A}(\nabla u, \nabla v) d\mu, \,\,\,\, &\Phi_r(u,v) = \int
_\Omega \textbf{b}(\nabla u) v d\mu + \int_\Omega c u v d\mu.
\end{align*}
We consider $\int_\Omega c u v d\mu$ as the action of the functional $c$ on the product $uv$. Since \textbf{A3} is satisfied, this action is well defined. In the same way, one can define $\tau(f,v) = \int_\Omega fv d\mu$. Then a function $u \in \Sob(\Omega)$ is a weak variational solution to (\ref{prob}) with $f \in H^{-1}(\Omega)$ if and only if
\begin{equation}
\label{weak_p}
\Phi_0(u,v) + \Phi_r(u,v) = \tau(f,v)\,\,\,\,\, \forall v \in \Sob(\Omega).
\end{equation}
By \textbf{A3} it follows that there exists a unique operator\footnote{Friedrichs extension is meant \cite{Kato}.} $\mathcal{A}$ generated by $\Phi$,  such that $\mathcal{A}u = f$. The operator $\mathcal{A}$ has a bounded inverse operator $\mathcal{R}: H^{-1}(\Omega) \to \Sob(\Omega)$.

In the fourth section we shall impose additional conditions on the coefficients of $\mathcal{A}$ (see (\ref{cond1})--(\ref{cond2})). 

Let $\textbf{g}(\cdot, \cdot)$ be the section of  $T^* M \times T^* M$ associated with the structure $g$ and $C_{emb}$ be the embedding constant of the continuous embedding $\Sob(\Omega) \hookrightarrow L_q(\Omega)$, $1/p + 1/q = 1/2$. Then, \textbf{A3} holds whenever $\textbf{b}\in L_p(\Omega)$,  $c \in W^{-1}_p(\Omega)$, and
$$
C_{emb}(\| \textbf{b} \|_{L_p(\Omega)} + \|c\|_{W^{-1}_p(\Omega)}) < \alpha,
$$
where $\alpha$ is a constant in \textbf{A1} and
$$
\| \textbf{b} \|_{L_\infty(\Omega)} =  \mathrm{ess}\sup_{x \in \Omega} |\textbf{g}(\textbf{b},\textbf{b})|^{1/2}, \,\,\,\, \| \textbf{b} \|_{L_p(\Omega)} = \left( \int_\Omega |\textbf{g}(\textbf{b},\textbf{b})|^{p/2} \right)^{1/p}, \,\,\,\, p \in [1,\infty).
$$

\subsection{Weak smoothness}

Further we shall use the following notion. Let $L$ be a (not necessarily compact) manifold, $E(L)$ be a real Banach space of functions $f: L \to \mathbb{R}$. Denote
$$
\tilde{E}(\Omega) \stackrel{\mathrm{def}}{=} \left\{ u\in E(L) \left|\,\, \mathrm{supp} u \subset \bar{\Omega} \right. \right\}.
$$

\subsubsection{Nikolskii spaces}

Let us denote $v_h(x) = v(x+h)$ for $v: \mathbb{R}^d \to \mathbb{R}$. We need to recall the definition of Nikolskii spaces $N^{k+\gamma}_p(\mathbb{R}^d)$, $\gamma \in (0,1]$, $k \in \mathbb{Z}_+$, $p \in [1, \infty]$:
\begin{align*}
N^{k+\gamma}_p (\mathbb{R}^d) &= \left\{ v \in W^k_p(\mathbb{R}^d) \left| \,\, \| v \|_{N_p^{k+\gamma}(\mathbb{R}^d)} \stackrel{\mathrm{def}}{=} \| v \|_{W^{k}_p(\mathbb{R}^d)} + [v]_{N^{k+\gamma}_p (\mathbb{R}^d)} \right.  \right\},\\
[v]_{N^{k+\gamma}_p (\mathbb{R}^d)} &\stackrel{\mathrm{def}}{=} \left\{ \begin{aligned} \max_{|\alpha| = k} \sup_{h \in \mathbb{R}^d, h \neq 0} \frac{\| \partial^{\alpha} v_h - \partial^\alpha v \|_{L_p(\mathbb{R}^d)}}{|h|^\gamma},\,\,\,\, \gamma \in (0,1)\\ \max_{|\alpha| = k} \sup_{h \in \mathbb{R}^d, h \neq 0} \frac{\| \partial^{\alpha} v_{2h} - 2 \partial^\alpha v_h + \partial^\alpha v \|_{L_p(\mathbb{R}^d)}}{|h|},\,\,\,\, \gamma = 1 \end{aligned} \right..
\end{align*}
Here $\alpha = (\alpha_1, \ldots ,\alpha_d) \in \mathbb{Z}_+^d$ are multi-indices, $\partial^\alpha = \partial^{\alpha_1}_{x_1} \cdots \partial^{\alpha_d}_{x_d}$, $|\alpha| = \sum_{i=1}^{d} \alpha_i$. Assume that $N^{k+\gamma}_{p,0}(\mathbb{R}^d) \subset N^{k+\gamma}_{p}(\mathbb{R}^d)$, $\gamma \in (0,1)$, is the space of all functions satisfying the condition\footnote{In  \cite{Muramatu}, it is denoted by $B^{k+\gamma}_{p,\infty-}(\mathbb{R}^d)$ (see Proposition \ref{duality})}
$$
\lim_{|h|\to 0, \infty} \max_{|\alpha| = k} \frac{\| \partial^{\alpha} v_h - \partial^\alpha v \|_{L_p(\mathbb{R}^d)}}{|h|^\gamma} = 0.
$$

Let $\mathcal{U} = \{(U,\kappa_U)\}$ be a certain fixed finite atlas of $M$ and $\{\psi_{U}\}$ be a corresponding smooth partition of unity. We define $N^{k+\gamma}_p(M)$ (or $u \in  N^{k+\gamma}_{p,0}(M)$) as the space of all functions $u\in L_p(M)$ such that for any $(U, \kappa_U)\in\mathcal{U}$ the product $\psi_U \cdot u$ belongs to $\tilde{N}^{k+\gamma}_p(U)$ (or $\psi_U \cdot u \in \tilde{N}^{k+\gamma}_{p,0}(U)$), and
$$
\| u \|_{N^{k+\gamma}_p (M)} = \sum_{U } \| \psi_U u \|_{\tilde{N}^{k+\gamma}_p(U)}, \,\,\,\,\, \| u \|_{N^{k+\gamma}_{p,0} (M)} = \sum_{U } \| \psi_U u \|_{\tilde{N}^{k+\gamma}_{p,0}(U)}.
$$
By Lemma 4.2 in \cite{Nikolskii53} it follows that the definition of the spaces $N^{k+\gamma}_p(M)$ is independent of atlas and partition of unity. 

Let us consider  a compact metric space $(X,\delta)$. For sets $A, B \subset X$ let
$$
\mathrm{dist}(A,B) \stackrel{\mathrm{def}}{=} \inf_{x\in A, y \in B} \delta(x,y).
$$

\begin{proposition}
\label{Brezis}
For any open set $\Omega \subset M$ the following embedding holds
$$
\Wobm(\Omega) \hookrightarrow \tilde{N}^m_p(\Omega), \,\,\,\, m=1,2, \,\,\,\, p \in [1,\infty].
$$
Moreover, for any function $v \in \Wob(\Omega)$, any map $(U, \kappa_U)$ of $M$, and any open set $V\Subset \Omega \cap U$ the following implication holds
\begin{align*}
\forall \varphi \in C(U)\, \forall h \in \mathbb{R}^d\!:\! | h | < \mathrm{dist}(V, \partial (\Omega \cap U)) \!\Rightarrow\! \| \varphi (v_h - v)\|_{L_p(V)} &\leq \| \varphi \|_{C(\bar{V})} C_V \| v \|_{\WobSm(\Omega)} |h|,\\
\| \varphi (v_h - v)\|_{\WobSm(V)} &\leq \| \varphi \|_{C(\bar{V})} \hat{C}_V \| v \|_{\WobSmd(\Omega)} |h|.
\end{align*}
Here $\mathrm{dist}$ is computed with respect to the metric $|\cdot|$ in $\kappa_U(\bar{U})$, and partial difference is defined with respect to the linear structure in the image of $\kappa_U$.
\end{proposition}

The statement above follows by Proposition IX.3 \cite{Brezis}: if $v \in W^1_p(U')$, $p \in [1,\infty]$, $U \Subset U' \subset \mathbb{R}^d$, then for any $h \in \mathbb{R}^d$, $|h| < \mathrm{dist}(U, \partial U')$ holds
$$
\| v_h - v \|_{L_p(U)} \leq |h| \| \nabla v  \|_{L_p(U')}.
$$
Furthermore, if $v \in W^1_p(\mathbb{R}^d)$, then
\begin{equation}
\label{brz}
\| v_h - v \|_{L_p(\mathbb{R}^d)} \leq |h| \| \nabla v  \|_{L_p(\mathbb{R}^d)}.
\end{equation}

\subsubsection{Besov spaces}

Let us recall the definition of Besov spaces (following \cite{Triebel}). Propositions \ref{interpol_1} and \ref{interpol_2} below follow from the similar propositions for domains with Lipschitz boundaries $\Omega \subset \mathbb{R}^d$ and the fact that for any simply--connected bounded domains $V_1, V_2 \subset \mathbb{R}^d$ there is a linear homeomorphism $K:\Sobm(V_2) \to \Sobm(V_1)$, $m=1,2$, $K:L_2(V_2) \to L_2(V_1)$, $K: u \mapsto u \circ K_0$, generated by $C^{1,1}$--diffeomorphism $K_0: V_1 \to V_2$. Here $K_0$ is defined in a larger open set $V \Supset V_1$.

Assume that $F$ is the Fourier transform, $M_0 = \{\xi \in \mathbb{R}^d \,\,| \,\,|\xi| \leq 2 \}$, and $M_j = \{\xi \in \mathbb{R}^d \,\,| \,\, 2^{j-1} \leq |\xi| \leq 2^{j+1} \}$ for $j \in \mathbb{N}$, $\mathcal{S}'$ is the space of distributions of moderate growth.

\begin{definition}
Let us define for $s\in \mathbb{R}$, $p \in (1,\infty)$ the following spaces
\begin{align*}
B^s_{p,q}(\mathbb{R}^d) = &\left\{ f\in \mathcal{S}'(\mathbb{R}^d) \,\,|\,\, f \stackrel{\mathcal{S}'}{=} \sum_{j = 0}^{\infty} a_j(x);\,\, \mathrm{supp} F a_j \subset M_j;\right.\\
&\left. \| \{ a_j \} \|_{l^s_q(L_p)} = \left[ \sum_{j = 0}^{\infty}(2^{sj} \| a_j \|_{L_p(\mathbb{R}^d)})^q \right]^{1/q} < \infty \right\},\,\,\,\,\,\, q \in [1,\infty)
\end{align*}
\begin{align*}
B^s_{p,\infty}(\mathbb{R}^d) = \left\{ f\in \mathcal{S}'(\mathbb{R}^d) \,\,|\,\, f \stackrel{\mathcal{S}'}{=} \sum_{j = 0}^{\infty} a_j(x);\,\, \mathrm{supp} F a_j \subset M_j;\right.\\
\left. \| \{ a_j \} \|_{l^s_\infty(L_p)} = \sup_{j \in \mathbb{Z}_+} 2^{sj}  \| a_j \|_{L_p(\mathbb{R}^d)} < \infty \right\},\,\,\,\,\,\, q = \infty
\end{align*}
with the norms
$$
\| f \|_{B^s_{p,q}(\mathbb{R}^d)} = \inf_{f \stackrel{\mathcal{S}'}{=} \sum_{j = 0}^{\infty} a_j} \| \{ a_j\} \|_{l^s_q(L_p)}.
$$
\end{definition}

The space $B^s_{p,q}(M)$ is defined by a finite atlas $\mathcal{U} = \{(U, \kappa_U)\}$ and subordinate partition of unity $\{\psi_U\}$, $\mathrm{supp} \psi_U \subset U$, in the following way. We suppose $u \in B^s_{p,q}(M)$, if $\psi_U u \in \tilde{B}^s_{p,q}(U)$; and introduce the norm
\begin{equation}
\label{bes_norm}
\| u \|_{B^s_{p,q}(M)} = \sum_{U} \| \psi_U u \|_{\tilde{B}^s_{p,q}(U)}.
\end{equation}

By the interpolation property of $B^s_{p,q}(\mathbb{R}^d)$, one obtains that the norm is independent on $\mathcal{U}$ and $\{\psi_U\}$, up to equivalence.

For $s \in \mathbb{R}_+ \backslash \mathbb{Z}_+$, Nikolskii and Sobolev--Slobodetskii spaces are the special cases of Besov spaces:
\begin{align*}
N^s_p(M) = B^s_{p,\infty}(M), \,\,\,\, \tilde{N}^s_p(\Omega) = \tilde{B}^s_{p,\infty}(\Omega),\\
W^s_p(M) = B^s_{p,p}(M), \,\,\,\, \tilde{W}^s_p(\Omega) = \tilde{B}^s_{p,p}(\Omega),
\end{align*}
and for any $p \in (1,\infty)$, $q \in [1,\infty]$, $s > \epsilon > 0$, the following chain of embeddings iholds
\begin{equation}
\label{chain_emb}
\tilde{B}^{s+\epsilon}_{p,\infty}(\Omega) \hookrightarrow \tilde{B}^{s}_{p,1}(\Omega) \hookrightarrow \tilde{B}^{s}_{p,q}(\Omega) \hookrightarrow \tilde{B}^{s}_{p,\infty}(\Omega) \hookrightarrow \tilde{B}^{s-\epsilon}_{p,1}(\Omega).
\end{equation}

Let us denote $B^s_{p,q}(\Omega) = B^s_{p,q}(M) / \left\{u \in B^s_{p,q}(M) \,\,\left|\,\, u|_\Omega \equiv 0 \right. \right\}$ with the norm
$$
\| v \|_{B^s_{p,q}(\Omega)} \stackrel{\mathrm{def}}{=} \inf_{u|_\Omega = v} \| u \|_{B^s_{p,q}(M)}
$$

Consider $(\cdot, \cdot)_{s,q}$ to be the real interpolation functor.

\begin{proposition}
\label{interpol_1}
For any $s\in (0,1)$, $q \in [1,\infty]$, and domain\footnote{Since $\partial M = \varnothing \in C^{0,1}$, in Propositions \ref{interpol_1}, \ref{interpol_2}, \ref{duality} one can suppose that domain $\Omega'$ coincides with~$M$} $\Omega' \subseteq M$ with a Lipschitz boundary, Besov spaces are the results of the following interpolation procedure:
\begin{align*}
\tilde{B}^{s}_{2,q}(\Omega') = (L_2(\Omega'), \Sob(\Omega'))_{s,q}, \,\,\,\, &\tilde{B}^{1 + s}_{2,q}(\Omega') = (\Sob(\Omega'), \Sobd(\Omega'))_{s,q},\\
B^{-s}_{2,q}(\Omega') = &(L_2(\Omega'), H^{-1}(\Omega'))_{s,q};
\end{align*}
in case of $\mathbb{R}^d$ the similar relations hold
\begin{align*}
B^{s}_{2,q}(\mathbb{R}^d) = (L_2(\mathbb{R}^d), H^1(\mathbb{R}^d))_{s,q}, \,\,\,\, &B^{1 + s}_{2,q}(\mathbb{R}^d) = (H^1(\mathbb{R}^d), H^2(\mathbb{R}^d))_{s,q},\\
B^{-s}_{2,q}(\mathbb{R}^d) = &(L_2(\mathbb{R}^d), H^{-1}(\mathbb{R}^d))_{s,q}.
\end{align*}
\end{proposition}

\begin{proposition}
\label{interpol_2}
For any $t, s \in (0,1)$, $0 < s_1 < s_2 < 1$, $q \in [1,\infty]$, and domain $\Omega' \subseteq M$, $\partial \Omega' \in C^{0,1}$, the following relations hold
\begin{align*}
(\Sob(\Omega'), \tilde{B}^{1+s}_{2,q}(\Omega'))_{t,2} = \tilde{H}^{1+ts}(\Omega'),\,\,\,\,\, &(H^{-1}(\Omega'), B^{-1+s}_{2,q}(\Omega'))_{t,2} = \tilde{H}^{1+ts}(\Omega'),\\
(B^{-s_1}_{2,q}(\Omega'), B^{-s_2}_{2,q}(\Omega'))_{t,2} = &H^{-(1-t)s_1 - t s_2}(\Omega');
\end{align*}
and in case of $\mathbb{R}^d$ one has
\begin{align*}
(H^1(\mathbb{R}^d), B^{1+s}_{2,q}(\mathbb{R}^d))_{t,2} = H^{1+ts}(\mathbb{R}^d),\,\,\,\,\, &(H^{-1}(\mathbb{R}^d), B^{-1+s}_{2,q}(\mathbb{R}^d))_{t,2} = H^{1+ts}(\mathbb{R}^d),\\
(B^{-s_1}_{2,q}(\mathbb{R}^d), B^{-s_2}_{2,q}(\mathbb{R}^d))_{t,2} = &H^{-(1-t)s_1 - t s_2}(\mathbb{R}^d).
\end{align*}
\end{proposition}

\begin{proposition}[by \cite{Muramatu}]
\label{duality}
Let $s \in \mathbb{R}_+$, $p \in (1,\infty)$, $q \in (1,\infty)$, $\Omega \subseteq M$ be a domain with a Lipschitz boundary then
$$
\left[ \tilde{N}^s_{p,0} (\Omega) \right]' = B^{-s}_{p',1}(\Omega), \,\,\,\, \left[ \tilde{B}^s_{p,q} (\Omega) \right]' = B^{-s}_{p',q'}(\Omega), \,\,\,\,  1/p + 1/p' = 1, \,\,\,\, 1/q + 1/q' = 1.
$$
\end{proposition}

\begin{proposition}[by \cite{Savare98}]
\label{ext}
Let $E_1$, $E_2$, $F$ be Banach spaces, an embedding $E_1 \hookrightarrow E_0$ be continuous, an operator  $\mathcal{T}: E_1 \to F$ be bounded. If there exists a constant $L > 0$ and a number $s \in (0,1)$ such that
$$
\| \mathcal{T} e \|_{F} \leq L \| e \|^{1-s}_{E_0} \| e \|^s_{E_1},\,\,\,\,\, \forall e \in E_1,
$$
then by continuity one can extend $\mathcal{T}$ as an operator from $(E_0, E_1)_{s,1}$ to $F$, and there is a constant $c_s$ depending only on $s$ for which
$$
\| \mathcal{T} \|_{(E_0, E_1)_{s,1} \to F} \leq c_s L.
$$
\end{proposition}

\section{Estimates of solutions to problem (\ref{prob})}

\begin{definition}
\label{semi}
A bilinear form $\zeta$ is called Hölder continuous of order $\gamma \in (0,1]$ in $H \subset \mathrm{dom} \zeta \subset L_1(\Omega)$, $\Omega \subset M$, if there exists an atlas $\mathcal{U}$ such that for any map $(U, \kappa_U) \in \mathcal{U}$ and any function $\varphi \in C^{1,1}(M)$ with $\mathrm{supp}\, \varphi \subset U$ there are constants  $C_{U,\varphi}$, $C^{H}_{\zeta}$ such that
$$
\forall u \in H \,\, \forall h \in \mathbb{R}^d \! \! :|h| < \mathrm{dist}(\mathrm{supp} \varphi, \partial U), \,z = \varphi (u - u_h) \in H \Rightarrow |\zeta(z)| \leq C_{U,\varphi} C^H_{\zeta} \| u \|_H |h|^{\gamma},
$$
where $\varphi u_h$  equals $\varphi(x) \cdot \left[ u\circ \kappa_U^{-1} \circ (x+h)\circ\kappa_U\right]$ for $x \in \mathrm{supp} \varphi$, and equals zero if $x \notin \mathrm{supp} \varphi$.
\end{definition}

\begin{theorem}
\label{base}
Assume that $u$ is a solution to the equation (\ref{prob}), $\Omega$ has a Hölder boundary of order $\gamma_\Omega \in (0,1]$, $\mathcal{A}$ satisfies conditions \textbf{A1}--\textbf{A3}, the linear forms $\Phi_r(u,\cdot)$, $\tau(f,\cdot)$ are Hölder continuous of order $\beta_0 \in (0,1]$ in $\Sob(\Omega)$. If $\gamma_0 = \min\{\gamma_c, \beta_0\},$ then
\begin{equation}
\label{est_1}
\| u \|^2_{\tilde{N}_2^{1 + \gamma_\Omega \gamma_0 / 2}(\Omega)} \leq C(\mathcal{A}, \Omega, M) \| u \|_{\SobSl(\Omega)} \left[\| u \|_{\SobSl(\Omega)} + C^{\SobSl(\Omega)}_{\tau(f,\cdot)} + C^{\SobSl(\Omega)}_{\Phi_r(u,\cdot)}\right],
\end{equation}
where $C(\mathcal{A}, \Omega, M) > 0$ depends only on $\mathcal{A}$, $\Omega$, $M$,  and $C^{\SobSl(\Omega)}_{\tau(f,\cdot)}$, $C^{\SobSl(\Omega)}_{\Phi_r(u,\cdot)}$ are the constants in Definition 3.1.

Further, if $u \in \tilde{N}^{1+s}_2(\Omega)$, $s \in (0,1/2)$, and the forms $\Phi_r(u,\cdot)$, $\tau(f,\cdot)$ are Hölder continuous of order $\beta_s \in (0,1]$ in $\tilde{N}^{1+s}_2(\Omega)$, $\gamma_s = \min\{\gamma_c, \beta_s\}$, then for the constants $C_{\tau(f,\cdot)}$, $C_{\Phi_r(u,\cdot)}$ in Definition \ref{semi}, one has
\begin{equation}
\label{est_2}
\| u \|^2_{\tilde{N}_2^{1 + \gamma_\Omega \gamma_s / 2}(\Omega)} \leq C(\mathcal{A}, \Omega, M) \left[\| u \|^2_{\SobSl(\Omega)} + \| u \|_{\tilde{N}_2^{1 + s}(\Omega)} \left( C^{\tilde{N}_2^{1 + s}(\Omega)}_{\tau(f,\cdot)} + C^{\tilde{N}_2^{1 + s}(\Omega)}_{\Phi_r(u,\cdot)} \right) \right].
\end{equation}
\end{theorem}

\begin{proof}
As the $\tilde{N}^s_2(\Omega)$-norm is independent of the atlas from its definition (up to  equivalence), we shall assume that this atlas coincides with atlas $\mathcal{V}$ in Definition~\ref{boundary}. Let $\{\psi_V\}$ be a subordinate (with respect to $\mathcal{V}$) partition of unity, $\mathrm{supp}\, \psi_V \subset V$. Consider the function $\phi(h) = |h| + C_\Omega |h|^{\gamma_\Omega}$, $h \in \mathbb{R}^d$. Then the functions $(\psi_V)_{h \pm \phi(h) \xi_V}$ are well defined for $|h| < \phi^{-1}\left[ \mathrm{dist}(\mathrm{supp}\psi_V, \partial V) / 20 \right]$.

Since for the proof it suffices to obtain (\ref{est_1})--(\ref{est_2}) with the left-hand side replaced by $\psi_V u$, $(V, \kappa_V) \in \mathcal{V}$, without the loss of generality, we can assume that a chart $V$ is fixed and for convenience we shall write  $\psi$, $\xi$ instead of $\psi_V$ and $\xi_V$. We estimate the difference
\begin{equation}
\label{est_3}
\| \psi \cdot (u - u_h)\|_{\SobSl(V)} \leq \| \psi(u - u_t) \|_{\SobSl(V)} + \| \psi(u_t - u_h) \|_{\SobSl(V)},
\end{equation}
where $u_{\pm t} = u_{\mp \phi(h) \xi}$. One can note that it is possible to rewrite the terms in the right hand-side of (\ref{est_3}) as $\| \psi (u_t - v) \|_{\SobSl(V)}$, where $v$ is equal to $u$ and $u_h$ for the first and the second terms respectively. The following inequality holds:
\begin{align*}
\| \psi(u_t - v)\|^2_{\SobSm(V)} = \int_M \mathfrak{g}(x, \nabla \left[\psi(u_t - v)\right]) d\mu = \int_M \mathfrak{g}(x, \nabla \left[\psi_{-t}(u - v_{-t})\right]) d\mu + \\
\int_M \mathfrak{g} (x + \phi(h)\xi, \nabla \left[\psi_{-t}(u - v_{-t})\right]) d\mu_{-t} - \int_M \mathfrak{g}(x, \nabla \left[\psi_{-t}(u - v_{-t})\right]) d\mu_{-t} + \\
\int_M \mathfrak{g}(x, \nabla \left[\psi_{-t}(u - v_{-t})\right]) (d\mu_{-t} - d\mu),
\end{align*}
here $\mathfrak{g}(x,\eta) = \textbf{G}_x(\eta, \eta)$. Due to smoothness of $M$, it is evident that
\begin{align*}
\int_M \mathfrak{g} (x + \phi(h)\xi, \nabla \left[\psi_{-t}(u - v_{-t})\right]) d\mu_{-t} - \int_M \mathfrak{g}(x, \nabla \left[\psi_{-t}(u - v_{-t})\right]) d\mu_{-t} \leq \\
C(M,\Omega) \| \psi \|_{C^{0,1}(M)} \| u \|^2_{\SobSm(\Omega)} |h|^{\gamma_\Omega},\\
\int_M \mathfrak{g}(x, \nabla \left[\psi_{-t}(u - v_{-t})\right]) (d\mu_{-t} - d\mu) \leq C(M,\Omega) \| \psi \|_{C^{0,1}(M)} \| u \|^2_{\SobSm(\Omega)} |h|^{\gamma_\Omega}.
\end{align*}
Thus
\begin{align*}
\| \psi(u_t - u) \|^2_{\SobSm(V)} \leq \| \psi(u - u_{-t}) \|^2_{\SobSm(V)} + C_{M,\Omega,\psi} \| u \|^2_{\SobSm(\Omega)} |h|^{\gamma_\Omega} + \tilde{C}_{M,\Omega, \psi}\| u \|^2_{\SobSm(\Omega)} |h|^{\gamma_\Omega},\\
\| \psi(u_t - u_h) \|^2_{\SobSm(V)} \leq \| \psi(u - (u_h)_{-t}) \|^2_{\SobSm(V)} + C_{M,\Omega,\psi} \| u \|^2_{\SobSm(\Omega)} |h|^{\gamma_\Omega} + \tilde{C}_{M,\Omega, \psi}\| u \|^2_{\SobSm(\Omega)} |h|^{\gamma_\Omega}.
\end{align*}

Let us define the following operator\cite{Savare02}:
$$
T^{\psi}_h u = \psi u_h + (1 - \psi) u,
$$
and introduce $\varphi_1(h) = \phi(h) \xi$, $\varphi_2(h) = h + \varphi_1(h)$. Since $\mathrm{supp}(u - T^{\psi}_{\varphi_i(h)}u) \subset \bar{\Omega}$ and $\tilde{H}^1(\Omega) = \Sob(\Omega)$ one can obtain that $(u - T^{\psi}_{\varphi_i(h)}u) \in \Sob(\Omega)$. 

Hence, from condition \textbf{A1} it follows that:
\begin{align*}
\| \psi(u - u_{-t}) \|^2_{\SobSm(\Omega)} \leq \frac{1}{\alpha} \Phi_0 \left(T^{\psi}_{\varphi_1(h)} u - u, T^{\psi}_{\varphi_1(h)} u - u \right),\\
\| \psi(u - (u_h)_{-t}) \|^2_{\SobSm(\Omega)} \leq \frac{1}{\alpha} \Phi_0 \left(T^{\psi}_{\varphi_2(h)} u - u, T^{\psi}_{\varphi_2(h)} u - u \right),
\end{align*}
\begin{equation}
\label{nine}
\Phi_0\left( T^{\psi}_{\varphi_i(h)} u - u, T^{\psi}_{\varphi_i(h)} u - u \right)  = \Phi_0 (T^{\psi}_{\varphi_i(h)} u, T^{\psi}_{\varphi_i(h)} u) - \Phi_0(u,u) + 2 \Phi_0(u, T^{\psi}_{\varphi_i(h)}u - u).
\end{equation}

Since the linear form $\tau(f, \cdot)$ is Hölder continuous of order $\beta_s$ in $\tilde{N}^{1+s}_2(\Omega)$, $s\in (0,1)$ and in $\Sob(\Omega)$, $s = 0$,  we have
\begin{align*}
|\tau(f, \psi(u - u_{\varphi_i(h)}))| &\leq C_{\Omega} C_{V,\psi} C^{\SobSm(\Omega)}_{\tau(f, \cdot)} \| u \|_{\SobSm(\Omega)} |h|^{\gamma_\Omega \beta_0}, \,\,\,\, s = 0,\\
|\tau(f, \psi(u - u_{\varphi_i(h)}))| &\leq C_{\Omega} C_{V,\psi} C^{\tilde{N}^{1+s}_2(\Omega)}_{\tau(f, \cdot)} \| u \|_{\tilde{N}^{1+s}_2(\Omega)} |h|^{\gamma_\Omega \beta_s}, \,\,\,\, s \in (0,1/2).
\end{align*}

Similarly, as the form $\Phi_r(u,\cdot)$ is Hölder continuous of order $\beta_s$, we obtain
\begin{align*}
|\Phi_r(u, \psi(u - u_{\varphi_i(h)}))| &\leq C_{\Omega} C_{V,\psi} C^{\SobSm(\Omega)}_{\Phi_r(u, \cdot)} \| u \|_{\SobSm(\Omega)} |h|^{\gamma_\Omega \beta_0}, \,\,\,\, s = 0,\\
|\Phi_r(u, \psi(u - u_{\varphi_i(h)}))| &\leq C_{\Omega} C_{V,\psi} C^{\tilde{N}^{1+s}_2(\Omega)}_{\Phi_r(u, \cdot)} \| u \|_{\tilde{N}^{1+s}_2(\Omega)} |h|^{\gamma_\Omega \beta_s}, \,\,\,\, s \in (0,1/2).
\end{align*}

It remains to estimate the terms $\Phi_0(T^{\psi}_{\varphi_i(h)}u, T^{\psi}_{\varphi_i(h)}u) - \Phi_0(u,u)$ in (\ref{nine}). Let us note that the gradient of $T^{\psi}_{\varphi_i(h)}u$ equals
$$
\nabla T^{\psi}_{\varphi_i(h)}u = (\psi \nabla u_{\varphi_i (h)} + (1-\psi)\nabla u) + (\nabla \psi) (u_{\varphi_i(h)} - u) = T^{\psi}_{\varphi_i (h)} \nabla u + (\nabla \psi) (u_{\varphi_i(h)} - u).
$$
For $\mathfrak{a}(x,\eta) = \textbf{A}_x(\eta, \eta)$ we see that
\begin{align}
\Phi_0(T^{\psi}_{\varphi_i(h)}u, T^{\psi}_{\varphi_i(h)}u) - \Phi_0(u,u) \leq \\
\label{eleven}
\int_M \mathfrak{a}(x, T^{\psi}_{\varphi_i (h)} \nabla u + (\nabla \psi) (u_{\varphi_i(h)} - u)) d\mu - \int_M \mathfrak{a}(x, T^{\psi}_{\varphi_i (h)} \nabla u) d\mu + \\
\label{twelve}
\int_M \mathfrak{a}(x, T^{\psi}_{\varphi_i (h)} \nabla u) d\mu - \int_M \mathfrak{a}(x, \nabla u) d\mu.
\end{align}
Due to condition \textbf{A2} and the Cauchy inequality we have
$$
\mathfrak{a}(x, \xi + \eta) - \mathfrak{a}(x,\xi) \leq \left( \mathfrak{a}(x, \eta) \mathfrak{a}(x, 2\xi + \eta) \right)^{1/2} \leq \| \textbf{A} \|_{C(M)} \mathfrak{g}(x, \eta)^{1/2} \left( 2 \mathfrak{g}(x,\xi)^{1/2} + \mathfrak{g}(x,\eta)^{1/2} \right).
$$
Thus for (\ref{eleven}) the following estimate holds
\begin{align*}
\int_M \mathfrak{a}(x, T^{\psi}_{\varphi_i (h)} \nabla u + (\nabla \psi) (u_{\varphi_i(h)} - u)) d\mu - \int_M \mathfrak{a}(x, T^{\psi}_{\varphi_i (h)} \nabla u) d\mu \leq \\
C_V \| \textbf{A} \|_{C(M)} \| \psi (u_{\varphi_i(h)} - u) \|_{L_2(V)} \left( \| \psi (u_{\varphi_i(h)} - u) \|_{L_2(V)} + 2 \| T^{\psi}_{\varphi_i (h)} \nabla u \|_{L_2(V)} \right).
\end{align*}
From Proposition \ref{Brezis} we conclude
$$
\| \psi (u_{\varphi_i(h)} - u) \|_{L_2(V)} \leq \tilde{C}_{V,M} C_\Omega |h|^{\gamma_\Omega} \| u \|_{\SobSl(\Omega)},
$$
and therefore we can obtain an upper estimate for (\ref{eleven}) as $C'_m C_{V,\psi} C_\Omega |h|^{\gamma_\Omega} \| u \|^2_{\SobSl(\Omega)}$. Since $\mathfrak{a}$ is convex, we have
\begin{align*}
\mathfrak{a}(x, T^{\psi}_{\varphi_i (h)} \nabla u) -\mathfrak{a}(x, \nabla u) \leq \left[T^{\psi}_{\varphi_i (h)} \mathfrak{a}(x,\nabla u)\right] - \mathfrak{a}(x,\nabla u) = \psi \left[ \mathfrak{a}(x, \nabla u_{\varphi_i(h)}) - \mathfrak{a}(x,\nabla u) \right];
\end{align*}
thus,
$$
\int_M \mathfrak{a}(x, T^{\psi}_{\varphi_i (h)} \nabla u) -\mathfrak{a}(x, \nabla u) d\mu \leq \int_M \psi \left[ \mathfrak{a}(x, \nabla u_{\varphi_i(h)}) - \mathfrak{a}(x,\nabla u) \right] d\mu,
$$
so, by Hölder continuity of \textbf{A} we have the following estimate for sum (\ref{twelve}):
\begin{align*}
\int_M \psi \left[ \mathfrak{a}(x, \nabla u_{\varphi_i(h)}) - \mathfrak{a}(x,\nabla u) \right] d\mu =\\
\int_M \psi_{-\varphi_i(h)} \mathfrak{a}(x - \varphi_i(h), \nabla u) d\mu_{-\varphi_i(h)} - \int_M \psi  \cdot \mathfrak{a}(x, \nabla u) d\mu \leq \\
\int_M (\psi_{-\varphi_i(h)} - \psi) \mathfrak{a}(x - \varphi_i(h), \nabla u) d\mu_{-\varphi_i(h)} + \int_M \psi \cdot  ( \mathfrak{a}(x - \varphi_i(h), \nabla u) - \mathfrak{a}(x, \nabla u) )d\mu +\\
\int_M \psi \cdot \mathfrak{a}(x - \varphi_i(h), \nabla u) (d\mu_{-\varphi_i(h)} - d\mu) \leq \| \textbf{A}\|_{C^{0,\gamma_c}(M)} C_\Omega C_V C_{\psi} |h|^{\gamma_\Omega \gamma_c} \| u \|^2_{\SobSm(\Omega)}.
\end{align*}
\end{proof}

\section{Conditions for Hölder continuity of linear forms}

\begin{lemma}
\label{regr}
We have the following inequalities
\begin{align}
\label{lem_1}
\| u - u_h \|_{N^{\gamma_1}_{2}(\mathbb{R}^d)} &\leq C_{\gamma_1, \gamma_2} |h|^{\gamma_2 - \gamma_1} \| u \|_{N^{\gamma_2}_{2}(\mathbb{R}^d)}, \,\,u \in N^{\gamma_2}_{2}(\mathbb{R}^d), \,\,\,\,0 <\gamma_1 < \gamma_2 < 1;   \\
\label{lem_2}
\| u - u_h \|_{N^{\gamma_1}_{2}(\mathbb{R}^d)} &\leq C_{\gamma_1, 1} |h|^{1-\gamma_1} \| u \|_{H^1(\mathbb{R}^d)}, \,\,\,\, \,\,\,\,\,u \in H^{1}(\mathbb{R}^d), \,\,\,\,\,\,\,0 < \gamma_1 < 1;\\
\label{lem_3}
\| u - u_h \|_{N^{1 - \gamma_1}_{2}(\mathbb{R}^d)} &\leq \tilde{C}_{\gamma_1, 1} |h|^{\gamma_1 + \gamma_2} \| u \|_{N^{1+\gamma_2}_2(\mathbb{R}^d)},\, u \in N^{1+\gamma_2}_2(\mathbb{R}^d), \,0 < \gamma_1 < \gamma_1 + \gamma_2 \leq 1.
\end{align}
\end{lemma}
\begin{proof}
In fact, from estimate (\ref{brz}) we have
\begin{align}
\label{lem_p1}
\| u - u_h \|_{H^1(\mathbb{R}^d)} &\leq 2 \| u \|_{H^1(\mathbb{R}^d)};\\
\label{lem_p2}
\| u - u_h \|_{L_2(\mathbb{R}^d)} &\leq  |h| \| u \|_{H^1(\mathbb{R}^d)};\\
\label{lem_p3}
\| u - u_h \|_{H^1(\mathbb{R}^d)} &\leq  |h| \| u \|_{H^2(\mathbb{R}^d)}.
\end{align}
By the real interpolation of (\ref{lem_p1}), (\ref{lem_p2}) and from Proposition \ref{interpol_1} one can obtain
$$
\| u - u_h \|_{N^{\gamma_1}_{p,0}(\mathbb{R}^d)} = \| u - u_h \|_{(L_2(\mathbb{R}^d), H^1(\mathbb{R}^d))_{\gamma_1, \infty}} \leq C_{\gamma_1,1}' \cdot 2^{\gamma_1} |h|^{1-\gamma_1} \| u \|_{H^1(\mathbb{R}^d)}.
$$
In the same way, from Propositions \ref{interpol_1}, \ref{interpol_2}, by real interpolation of (\ref{lem_2}) and (\ref{lem_p3}) we conclude that (\ref{lem_3}) is true. Indeed, for $0 < \gamma_1 < 1 - \gamma_2$, we have
\begin{align*}
\| u - u_h \|_{N^{st + 1 - t}_2(\mathbb{R}^d)} &= \| u - u_h \|_{(H^1(\mathbb{R}^d), N^s_2(\mathbb{R}^d))_{t,\infty}} \leq \\
C_{t,s} |h|^{(1-s)t} |h|^{1-t} \| u \|_{(H^2(\mathbb{R}^d), H^1(\mathbb{R}^d))_{t,\infty}} &= C_{t,s} |h|^{1-st} \| u \|_{N^{2-t}_2(\Omega)},
\end{align*} 
where $t = 1- \gamma_2$, $s = 1 - \frac{\gamma_1}{1 - \gamma_2}$. As above, by real interpolation of (\ref{lem_p2}) and (\ref{lem_p3}) in case of $\gamma_1 = 1 - \gamma_2$ we infer that  (\ref{lem_3}) is true. Since $\mathcal{T}_h$ is a linear continuous operator in $N^s_2(\mathbb{R}^d)$, the following holds
\begin{equation}
\label{lem_p4}
\| u - u_h \|_{N^s_2(\mathbb{R}^d)} \leq 2 \| u \|_{N^s_2(\mathbb{R}^d)},
\end{equation}
thus,
$$
\| u - u_h \|_{N^{\gamma_1}_2(\mathbb{R}^d)} \leq C_{\gamma_1, t} |h|^{(1-\gamma_1) t} \| u \|_{(N^{\gamma_1}_2(\mathbb{R}^d), H^1(\mathbb{R}^d))_{t, \infty}} = C_{\gamma_1, t} |h|^{(1-\gamma_1)t} \| u \|_{N^{\gamma_1(1-t) + t}_2(\mathbb{R}^d)},
$$
here $t = \frac{\gamma_2 - \gamma_1}{1 - \gamma_1}$.
\end{proof}

\begin{corollary}
\label{right_side}
For any function $f \in \left[ \tilde{N}^{1-s}_{2,0}(\Omega) \right]'$, $s \in (0,1)$, the linear form $\tau(f, \cdot)$ is Hölder continuous of order $(s + t)$ in the space $\tilde{N}^{1+t}_{2}(\Omega)$, $s < s + t \leq 1$, and of order $s$ in the space $\Sob(\Omega)$, and
$$
C_{\tau(f,\cdot)}^{\tilde{N}^{1+t}_{2}(\Omega)} = C_{\tau(f,\cdot)}^{\SobSl(\Omega)} = \| f \|_{\left[ \tilde{N}^{1-s}_{2,0}(\Omega) \right]'}.
$$
If $f \in L_2(\Omega)$, the linear form $\tau(f,\cdot)$ is Hölder continuous of order $1$ in the space $\Sob(\Omega)$, and $C_{\tau(f,\cdot)}^{\SobSl(\Omega)} = \| f\|_{L_2(\Omega)}$.
\end{corollary}
\begin{proof}
Indeed, for an arbitrary map $(U,\kappa_U)$ from the atlas of manifold $M$ and function $\chi \in C^{1,1}(U)$, $\mathrm{supp} \chi \subset U$, $t> 0$, using estimate (\ref{lem_3}), one can obtain:
\begin{align}
|\tau(f, \chi(u - u_h))| \leq \| f \|_{\left[ \tilde{N}^{1-s}_{2,0}(\Omega) \right]'} \| \chi (u - u_h) \|_{\tilde{N}^{1-s}_{2,0}(\Omega)} \leq \\
C_M (\| \chi \|_{C^{0,1}(M)} + 1) \| f \|_{\left[ \tilde{N}^{1-s}_{2,0}(\Omega) \right]'} \| \chi u - (\chi u)_h \|_{\tilde{N}^{1-s}_{2,0}(\Omega)} \leq \\
\label{tt}
C_{M,\chi} \| f \|_{\left[ \tilde{N}^{1-s}_{2,0}(\Omega) \right]'} \| u \|_{\tilde{N}^{1+t}_{2}(\Omega)} |h|^{s+t}.
\end{align}
Now we must only prove that the form $\tau(f,\cdot)$ is Hölder continuous of order $s$. This follows by combining inequality (\ref{tt}) and estimate (\ref{lem_2}) of Lemma \ref{regr}.
\end{proof}

The following embedding Theorem is proved in \cite{Besov}.

\begin{theorem}
\label{emb}
Let $1 \leq p_0 \leq p_1 \leq \infty$, $0 < s_1 \leq s_0 < 1$. Then
\begin{align*}
N^{s_0}_{p_0} (\mathbb{R}^d) \hookrightarrow N^{s_1}_{p_1} (\mathbb{R}^d), \,\,\,\, s_0 - \frac{d}{p_0} = s_1 - \frac{d}{p_1};\\
N^{s_0}_{p_0} (\mathbb{R}^d) \hookrightarrow L_{p_2} (\mathbb{R}^d), \,\,\,\, s_0 - \frac{d}{p_0} > - \frac{d}{p_2}.
\end{align*}
One can replace in the formula above $\mathbb{R}^d$ with $M$.
\end{theorem}

It follows from Theorem 1 in \cite{Goldman} that one can obtain the following
$$
N^{s_0}_{p_0} (\mathbb{R}^d) \not\hookrightarrow L_{p_2} (\mathbb{R}^d), \,\,\,\, s_0 - \frac{d}{p_0} = - \frac{d}{p_2}.
$$

\begin{lemma}
\label{prod}
Let $u \in N^\alpha_2(M)$, $v \in N^\beta_2(M)$, $w \in L_2(M)$, $0 < \alpha \leq \beta < 1$. Then for any  $\varepsilon' > 0$ the products $uv$, $wv$ belong to $N^\alpha_{s-\varepsilon}(M)$ and $L_{s-\varepsilon}(M)$ respectively, where
$$
s = \frac{d}{d - \beta}.
$$
\end{lemma}

\begin{proof}
Since the multiplication and embedding operators
$$
L_2(M) \times L_2(M) \to L_1(M), \,\,\,\,\, N^{\beta}_2 (M) \hookrightarrow N^{\alpha}_2(M)
$$
are continuous, for $u \in N^\alpha_2(M)$, $v \in N^{\beta}_2(M)$ it clearly follows that the product $uv$ belongs to $N^\alpha_s(M)$, $s = 1$. Let us refine the order of summability $s$. Without loss of generality, we can assume that the supports of all functions $u$, $v$ are contained in a subdomain of a fixed chart $U$. Hence, the shift operator is well defined. Therefore
$$
\sup_{h \neq 0} \frac{\| uv - u_h v_h \|_{L_s(M)}}{|h|^\alpha} \leq \sup_{h \neq 0} \frac{\| ( u - u_h) v_h \|_{L_s(M)}}{|h|^\alpha} + \sup_{h \neq 0} \frac{\| u (v - v_h) \|_{L_s(M)}}{|h|^\alpha},
$$
and by Hölder's inequality it follows that
\begin{align}
\label{prod_1}
\frac{\| ( u - u_h) v_h \|_{L_s(M)}}{|h|^\alpha} \leq \frac{1}{|h|^\alpha} \| u - u_h \|_{L_{p_1 s}(M)} \| v_h \|_{L_{q_1 s}(M)};\\
\label{prod_2}
\frac{\| u (v - v_h) \|_{L_s(M)}}{|h|^\alpha} \leq \frac{1}{|h|^\alpha} \| u \|_{L_{p_2 s}(M)} \| v - v_h \|_{L_{q_2 s}(M)},
\end{align}
where $p_j, q_j \geq 1$, $1/p_j + 1/q_j = 1$, $j = 1,2$. For the boundedness of the right hand-side of (\ref{prod_1}) as $|h| \to 0$ it is sufficient to consider the case $p_1 s_1 = 2$. From Theorem \ref{emb}, for any $\varepsilon > 0$, we have the embedding $N^\beta_2(M) \hookrightarrow L_{q_1 s_1 - \varepsilon}(M)$, and
$$
\frac12 - \frac{1}{q_1 s_1} = \delta, \,\,\,\,\, \frac{1}{p_1 s_1} + \frac{1}{q_1 s_1} = \frac{1}{s_1}, \,\,\,\, \delta = \min\{ (1 - \varepsilon) / 2, \beta / d \}.
$$
Similarly, from Theorem \ref{emb} it follows that $N^\beta_2(M) \hookrightarrow N^{\alpha}_{q_2 s_2}(M)$, $N^{\alpha}_{2}(M)\hookrightarrow L_{p_2 s_2-\varepsilon}(M)$ and
$$
\frac12 - \frac{1}{q_2 s_2} = \frac{\beta - \alpha}{d}, \,\,\,\,\, \frac{1}{2} - \frac{1}{p_2 s_2} = \frac{\alpha}{d}, \,\,\,\,\, \frac{1}{p_1 s_1} + \frac{1}{q_1 s_1} = \frac{1}{s_1}.
$$
By setting $s = \min\{s_1, s_2\}$ we obtain the required.
\end{proof}

\begin{lemma}[see \cite{Grisvard}]
\label{prod_Zol}
For any function $u \in L_2(M)$, $v \in H^1(M)$, and $\varepsilon > 0$ the product $uv$ belongs to $L_s(M)$, $s = \min\{ 2 - \varepsilon, \frac{d}{d-1} \}$.
\end{lemma}

\begin{lemma}
\label{cond_coef}
The linear form $\Phi_r(u, \cdot)$, $u \in \Sob(\Omega)$, is Hölder continuous of order $\beta_t = \gamma + t \in (0,1]$, $\gamma \in (0,1]$, $t \in [0,1)$, in Nikolskii space $\tilde{N}^{1+t}_2(\Omega)$, $0 < t <1$ ($\Sob(\Omega)$ if $t = 0$) if for some number $\varepsilon > 0$ the following conditions hold:
\begin{enumerate}
\item Let $\gamma \in (0,1)$ then
\begin{equation}
\label{cond1}
\textbf{b} \in L_{\frac{d+\varepsilon}{1-\gamma}}(\Omega), \,\,\,\, c \in \left[\tilde{N}^{1-\gamma}_{\frac{d}{d-1} - \varepsilon, 0}(\Omega) \right]',
\end{equation}
and $C^{\tilde{N}^{1+t}_2(\Omega)}_{\Phi_r(u,\cdot)} = C^{\SobSl(\Omega)}_{\Phi_r(u,\cdot)} = \left( \| \textbf{b} \|_{L_{\frac{d+\varepsilon}{1-\gamma}}(\Omega)} + \| c \|_{\left[\tilde{N}^{1-\gamma}_{\frac{d}{d-1}-\varepsilon, 0}(\Omega) \right]'} \right) \| u \|_{\SobSl(\Omega)}$;
\item If $\gamma = 1$, $t = 0$, then
\begin{equation}
\label{cond2}
\textbf{b} \in L_{\infty}(\Omega), \,\,\,\, c \in L_{\max\{2 + \varepsilon, d\}}(\Omega),
\end{equation}
and $C^{\SobSl(\Omega)}_{\Phi_r(u,\cdot)} = \left( \| \textbf{b} \|_{L_{\infty}(\Omega)} + \| c \|_{L_{\max\{2 + \varepsilon, d\}}(\Omega)} \right) \| u \|_{\SobSl(\Omega)}$.
\end{enumerate}
(Constants $C^{\tilde{N}^{1+t}_2(\Omega)}_{\Phi_r(u,\cdot)}$, $C^{\SobSl(\Omega)}_{\Phi_r(u,\cdot)}$ are introduced in Definition \ref{semi}.)
\end{lemma}

\begin{proof}
Let us set $q_b = \frac{d+\varepsilon}{1 - \gamma}$, $1/q_b + 1/p_b = 1$, $p_c = \frac{d}{d-1} - \varepsilon$ and consider an arbitrary function $\chi \in C^{1,1}(U)$, $\mathrm{supp} \chi \subset U$, $(U, \kappa_U)$ is a map from the atlas of manifold $M$. From condition (\ref{cond1}) using Lemmas \ref{regr} and \ref{prod} one can obtain
\begin{align*}
\left| \int_{\Omega} \textbf{b} (\nabla v) \chi (u - u_h) d\mu \right| \leq \| \textbf{b} \|_{L_{q_b}(\Omega)} \cdot \left\| \textbf{G}(\nabla v, \nabla v)^{1/2}   \chi (u - u_h) \right\|_{L_{p_b}(\Omega)} \leq\\
C_U (1 + \|\chi \|_{C^{0,1}(U)}) \| \textbf{b} \|_{L_{q_b}(\Omega)} \| v \|_{\SobSm(\Omega)} \| \chi u - (\chi u)_h \|_{\tilde{N}^{1 - \gamma}_{2, 0}(\Omega)} \leq \\
C_{U,\chi} \| \textbf{b} \|_{L_{q_b}(\Omega)} |h|^{\gamma + t} \| v \|_{\SobSm(\Omega)} \| u \|_{N^t},
\end{align*}
where $N^t = \tilde{N}^{1+t}_2(\Omega)$ if $t \in (0,1)$, and $N^0 = \Sob(\Omega)$. In analogous way
\begin{align*}
\left| \int_{\Omega} c  v \chi (u - u_h) d\mu \right| \leq \| c \|_{\left[\tilde{N}^{1-\gamma}_{p_c, 0}(\Omega) \right]'} \cdot \left\| v \chi (u - u_h) \right\|_{\tilde{N}^{1-\gamma}_{p_c, 0}(\Omega)} \leq\\
C_U \left(1 + \|\chi \|_{C^{0,1}(U)}\right) \| c \|_{\left[\tilde{N}^{1-\gamma}_{p_c, 0}(\Omega) \right]'} \| v \|_{\SobSm(\Omega)} \| \chi u - (\chi u)_h \|_{\tilde{N}^{1 - \gamma}_{2, 0}(\Omega)} \leq \\
C_{U,\chi} \| c \|_{\left[\tilde{N}^{1-\gamma}_{p_c, 0}(\Omega) \right]'} |h|^{\gamma + t} \| v \|_{\SobSm(\Omega)} \| u \|_{N^t}.
\end{align*}

As above, using Lemma \ref{prod_Zol} and Proposition \ref{Brezis} one can conclude that
\begin{align*}
\left| \int_{\Omega} \textbf{b} (\nabla v) \chi (u - u_h) d\mu \right| \leq \| \textbf{b} \|_{L_{\infty}(\Omega)} \cdot \left\| \textbf{G}(\nabla v, \nabla v)^{1/2}   \chi (u - u_h) \right\|_{L_{1}(\Omega)} \leq\\
C_U \left(1 + \|\chi \|_{C^{0,1}(U)}\right) \| \textbf{b} \|_{L_{\infty}(\Omega)}\| v \|_{\SobSm(\Omega)} \| \chi u - (\chi u)_h \|_{L_2(\Omega)} \leq \\
C_{U,\chi} \| \textbf{b} \|_{L_{\infty}(\Omega)} |h|^{\gamma + t} \| v \|_{\SobSm(\Omega)} \| u \|_{\SobSm(\Omega)},
\end{align*}
and
\begin{align*}
\left| \int_{\Omega} c  v \chi (u - u_h) d\mu \right| \leq \| c \|_{L_{q_c}(\Omega)} \cdot \left\| v \chi (u - u_h) \right\|_{L_{p_c}(\Omega)} \leq\\
C_U \left(1 + \|\chi \|_{C^{0,1}(U)}\right) \| c \|_{L_{q_c}(\Omega)} \| v \|_{\SobSm(\Omega)} \| \chi u - (\chi u)_h \|_{L_2(\Omega)} \leq \\
C_{U,\chi} \| c \|_{L_{q_c}(\Omega)} |h| \| v \|_{\SobSm(\Omega)} \| u \|_{\SobSm(\Omega)},
\end{align*}
where $1/q_c + 1/p_c = 1$.
\end{proof}

\section{Savar\'{e}-type theorems}
\label{Results}

Taking into account the Propositions from Sections 3 and 4 let us prove the following

\begin{theorem}
Assume that $M$ is a $C^{1,1}$--smooth compact Riemannian manifold without boundary, $\Omega \subset M$ be a subdomain, $\partial\Omega \in C^{0,\gamma_\Omega}$, operator $\mathcal{A}$ satisfies \textbf{A1}--\textbf{A3}, and for $\gamma = \gamma_c \in (0,1)$ and $\gamma = \gamma_c = 1$ conditions (\ref{cond1}), and (\ref{cond2}) respectively, are fulfilled. Then the operator solving the problem (\ref{prob})
\begin{equation}
\label{th_51}
\begin{aligned}
\mathcal{R}: \left(H^{-1}(\Omega), \left(H^{-1}(\Omega), \left[ \tilde{N}^{1-\gamma_c}_{2,0}(\Omega) \right]'\right)_{1/2, 1}\right)_{t,2} \to \tilde{H}^{1+\gamma_\Omega \gamma_c t / 2}(\Omega), \,\,\,\, t\in (0,1), \,\,\, \gamma_c \in (0,1);\\
\mathcal{R}: \left(H^{-1}(\Omega), \left(H^{-1}(\Omega),  L_{2}(\Omega)\right)_{1/2, 1}\right)_{t,2} \to \tilde{H}^{1+\gamma_\Omega \gamma_c t / 2}(\Omega), \,\,\,\, t\in (0,1), \,\,\, \gamma_c = 1
\end{aligned}
\end{equation}
is continuous.
\end{theorem}

\begin{proof}
Let us use Theorem \ref{base} and estimate the constants  $C^{\SobSl(\Omega)}_{\tau(f,\cdot)}$, $C^{\SobSl(\Omega)}_{\Phi_r(u,\cdot)}$ using Corollary  \ref{right_side} and Lemma \ref{cond_coef}. Then
\begin{align*}
\| u \|^2_{\tilde{N}_2^{1 + \gamma_c \gamma_\Omega / 2}(\Omega)} \leq C \| u \|_{\SobSm(\Omega)} \left( \| u \|_{\SobSm(\Omega)} + \| f \|_{\left[\tilde{N}^{1-\gamma_c}_{2,0}(\Omega) \right]'} + \| u \|_{\SobSm(\Omega)} \right) \leq\\
C' \| f \|_{H^{-1}(\Omega)} \| f \|_{\left[\tilde{N}^{1-\gamma_c}_{2,0}(\Omega) \right]'}.
\end{align*}
On the one hand, it follows from Proposition \ref{ext} that the operator
$$
\mathcal{R}: \left(H^{-1}(\Omega), \left[ \tilde{N}^{1-\gamma_c}_{2,0}(\Omega) \right]'\right)_{1/2, 1} \to \tilde{N}^{1 + \gamma_\Omega \gamma_c / 2}_2(\Omega) \subset N^{1 + \gamma_\Omega \gamma_c / 2}_2(M)
$$
is bounded. On the other hand,
$$
\mathcal{R}: H^{-1}(\Omega) \to \Sob(\Omega) \subset H^1(M),
$$
therefore, applying Propositions \ref{interpol_1}, \ref{interpol_2}, one can obtain that $\mathcal{R}$ is bounded as an operator from the space $\left(H^{-1}(\Omega), \left(H^{-1}(\Omega),  L_{2}(\Omega)\right)_{1/2, 1}\right)_{t,2}$ to the space 
$$
(H^1(M), B^{1+\gamma_\Omega \gamma_c / 2}_{2,\infty}(M))_{t,2} = H^{1+\gamma_\Omega \gamma_c t /2}(M),\,\,\,\,\,\,\,  t \in (0,1).
$$
\end{proof}

Now, let us prove Theorem \ref{result}. Due to the embedding chain (\ref{chain_emb}), for any $\varepsilon > 0$ from (\ref{cond1}), there exists $\epsilon > 0$ such that the following is true:
$$
\tilde{N}^{1-\gamma_c}_{\frac{d}{d-1}- \varepsilon, 0}(\Omega) \hookrightarrow \tilde{N}^{1-\gamma_c - \epsilon / 2}_{\frac{d}{d-1},0}(\Omega) \hookrightarrow \tilde{W}^{1 - \gamma_c - \epsilon}_{\frac{d}{d-1}}(\Omega), \,\,\,\,\, \gamma_c \in (0,1),
$$
and one can choose $\epsilon \to 0$ as $\varepsilon \to 0$. Thus there is a linear bounded operator 
$$
T: \left( \tilde{W}^{1 - \gamma_c - \epsilon}_{\frac{d}{d-1}}(\Omega) \right)' \to \left( \tilde{N}^{1-\gamma_c}_{\frac{d}{d-1}- \varepsilon,0}(\Omega) \right)',
$$
and hence the conditions from p. 1 of Lemma \ref{cond_coef} are fulfilled if $c \in W^{-1 +\gamma_c + \epsilon}_{d}(\Omega)$.

Similarly, from Proposition \ref{duality} it follows that operator 
$$
S_\Omega: B_{2,1}^{-1+\gamma_c}(M) = \left[ \tilde{N}^{1-\gamma_c}_{2,0}(M) \right]' \to \left[ \tilde{N}^{1-\gamma_c}_{2,0}(\Omega) \right]', \,\,\,\,\,\Omega \subset M,
$$
is well defined. Using Propositions \ref{interpol_1} and \ref{interpol_2} we conclude
$$
B^{-1+\gamma_c / 2}_{2,1}(M) = (H^{-1}(M), B^{-1+\gamma_c}_{2,1}(M))_{1/2, 1} \stackrel{S_\Omega}{\to} (H^{-1}(\Omega), B^{-1+\gamma_c}_{2,1}(\Omega))_{1/2, 1},
$$
hence, operator
$$
\mathcal{R}: \left( H^{-1}(M), B^{-1 + \gamma_c / 2}_{2,1}(M) \right)_{t,2} = H^{-1 + \gamma_c t /2}(M) \to \tilde{H}^{1 + \gamma_c \gamma_\Omega t / 2}(\Omega), \,\,\,\, t \in (0,1).
$$
is bounded.

We conclude with the following generalization of Theorem \ref{result}.

\begin{theorem}
\label{th52}
Let $M$ be a $C^{1,1}$--smooth compact Riemannian manifold without boundary, $\Omega \subsetneq M$ be a subdomain, $C^{0,\gamma_\Omega}$, operator $\mathcal{A}$ satisfies to conditions \textbf{A1}--\textbf{A3}, $\textbf{b} \in L_{\frac{d+\varepsilon}{1 - \gamma_0}}(\Omega)$, $c \in \left[ \tilde{N}^{1-\gamma_0}_{\frac{d}{d-1}-\varepsilon,0} (\Omega) \right]'$, $\gamma_0 \in \left( 0, \gamma_c \right]$. Then operator $\mathcal{R}$ solving the problem (\ref{prob}) is bounded with respect to the following pairs
\begin{align}
\label{th52_1}
H^{-1 + \frac{2^{n-1} - 1}{2^{n-1}}\gamma_0 + \frac{1}{2^n}\gamma_0 s}(M) &\to \tilde{H}^{1 + \frac{\gamma_\Omega}{2 - \gamma_\Omega} \frac{2^{n-1} - \gamma_\Omega^{n-1}}{2^{n-1}}\gamma_0 + \left(\frac{\gamma_\Omega}{2}\right)^{n} \gamma_0 s}(\Omega), \,\,\, s \in (0,1)\\
\label{th52_2}
B^{-1 + \frac{2^{n} - 1}{2^{n}}\gamma_0}_{2,1}(M) &\to \tilde{N}_2^{1 + \frac{\gamma_\Omega}{2 - \gamma_\Omega} \frac{2^{n} - \gamma_\Omega^{n}}{2^{n}}\gamma_0}(\Omega),
\end{align}
and $n\in\mathbb{N}$ if $\gamma_0 \leq \gamma_c(1 - \gamma_\Omega / 2)$. Otherwise if there exists $N \in \mathbb{N}$ such that
$$
\gamma_c \geq \frac{2}{2-\gamma_\Omega} \frac{2^N - \gamma_\Omega^N}{2^N} \gamma_0,
$$
then we have the boundedness of $\mathcal{R}$ with respect to 
\begin{align}
\label{th52_3}
H^{-1 + \frac{2^{N} - 1}{2^{N}}\gamma_0 + r_N s}(M) &\to \tilde{H}^{1 + \frac{\gamma_\Omega}{2 - \gamma_\Omega} \frac{2^{N} - \gamma_\Omega^{N}}{2^{N}}\gamma_0 (1 - s) + \frac{\gamma_\Omega \gamma_c}{2}s}(\Omega), \,\,\,\, s \in (0,1),\\
\label{th52_4}
B_{2,1}^{-1 + \frac{2^{N} - 1}{2^{N}}\gamma_0 +  r_N}(M) &\to \tilde{N}_2^{1 + \frac{\gamma_\Omega \gamma_c}{2}}(\Omega),
\end{align}
where $r_N = \frac{1}{2^{N+1}}\gamma_0 + \gamma_\Omega \left( \frac{\gamma_c}{2} - \frac{1}{2-\gamma_\Omega} \frac{2^{N+1} - \gamma_\Omega^{N+1}}{2^{N+1}} \gamma_0 \right)$.
\end{theorem}

\begin{proof}
It is clear that (\ref{th52_1}), (\ref{th52_3}) come from the real interpolation of (\ref{th52_2}) and (\ref{th52_4}) respectively. Thus we must only check the boundedness in pairs (\ref{th52_2}), (\ref{th52_4}). From the proof of Theorem \ref{result}, it follows that operator $\mathcal{R}:B^{-1+\gamma_0/2}_{2,1}(M) \to \tilde{N}_2^{1+\gamma_\Omega \gamma_0 / 2}(\Omega)$ is continuous. Therefore the solution of (\ref{prob}) belongs to  $\tilde{N}_2^{1+\gamma_\Omega \gamma_0 / 2}(\Omega)$, moreover
$$
\| u \|_{\tilde{N}_2^{1+\gamma_\Omega \gamma_0 / 2}(\Omega)} \leq C_1 \| f \|_{B_{2,1}^{1+\gamma_0 / 2}(M)};
$$
and we can apply Theorem \ref{base}. From Corollary \ref{right_side} and Proposition \ref{cond_coef} it follows that linear forms $\tau(f, \cdot)$, $\Phi_r(u,\cdot)$ are Hölder continuous of order $\beta_s = \gamma_0 + s$, $s = \gamma_\Omega \gamma_0 /2$ in  the space $\tilde{N}_2^{1+\gamma_\Omega \gamma_0 / 2}(\Omega)$ if  $u \in \Sob(\Omega)$ and $f \in B_{2,1}^{-1 + \gamma_0}(M)$. Thus we apply Theorem \ref{base} again and for $\beta_s\leq \gamma_c$ we have
\begin{align*}
\| u \|^2_{\tilde{N}_2^{1 + \beta_s \gamma_\Omega / 2}(\Omega)} = \| u \|^2_{\tilde{N}_2^{1 + \gamma_0 \left( \frac{\gamma_\Omega}{2} + \left( \frac{\gamma_\Omega}{2} \right)^2 \right)}(\Omega)} \leq \\
c \left( \| u \|_{\tilde{N}_2^{1+\gamma_\Omega \gamma_0 / 2}(\Omega)} \| f \|_{B^{-1+\gamma_0}_{2,1}(M)} + \| u \|_{\SobSm(\Omega)} \| u \|_{\tilde{N}_2^{1+\gamma_\Omega \gamma_0 / 2}(\Omega)} \right) \leq\\ 
\tilde{C} \| f \|_{B^{-1+\gamma_0/2}_{2,1}(M)} \| f \|_{B^{-1+\gamma_0}_{2,1}(M)}.
\end{align*}
Hence due to Proposition \ref{ext} there exists a bounded linear extension of $\mathcal{R}$ from 
$$
(B^{-1+\gamma_0/2}_{2,1}(M), B^{-1+\gamma_0}_{2,1}(M))_{1/2,1} = B^{-1+3 \gamma_0 / 4}_{2,1}(M)
$$ 
to $\tilde{N}_2^{1 + \beta_s \gamma_\Omega / 2}(\Omega)$. Thus the following estimate holds 
$$
\| u \|^2_{\tilde{N}_2^{1 + \gamma_0 \left( \frac{\gamma_\Omega}{2} + \cdots + \left( \frac{\gamma_\Omega}{2} \right)^n \right)}(\Omega)}\leq \hat{C}_n \| f \|^2_{B^{-1+\left( \frac{1}{2} + \cdots + \frac{1}{2^{n}} \right) \gamma_0}_{2,1}(M)},
$$
while $n \leq N$. Therefore the boundedness of $\mathcal{R}$ in pairs (\ref{th52_2}) is obtained. To justify (\ref{th52_4}), let us set $s + t = \gamma_\Omega \gamma_c$, $t = \gamma_0 \left(\frac{\gamma_\Omega}{2} + \cdots + \left( \frac{\gamma_\Omega}{2}\right)^N \right)$ and use Corollary \ref{right_side}, then
$$
\| u \|^2_{\tilde{N}^{1+\gamma_c \gamma_\Omega /2}_2(\Omega)} \leq C \| f \|_{B^{-1+\left( \frac{1}{2} + \cdots + \frac{1}{2^{N}} \right) \gamma_0}_{2,1}(M)} \| f \|_{B^{-1+ \left[\frac{\gamma_\Omega \gamma_c}{2} - \left( \frac{1}{2} + \cdots + \frac{1}{2^{N}} \right) \gamma_0 \right]}_{2,1}(M)}.
$$
Using Proposition \ref{ext} we conclude (\ref{th52_4}).
\end{proof}

\begin{corollary}
Let the conditions from Theorem \ref{th52} are fulfilled, $\gamma_0 = \frac{2 - \gamma_\Omega}{2} \gamma_c$. Then operator 
$$
\mathcal{R}: H^{-1 + \gamma_0 s}(M) \to \tilde{H}^{1 + \gamma_c \gamma_\Omega s /2}(\Omega), \,\,\,\, s \in [0,1)
$$
is continuous. 

In particular, if $\gamma_\Omega = 1$, \textbf{A1}--\textbf{A3} hold,  $\textbf{b} \in L_{\frac{2d}{2 - \gamma_c}}(\Omega)$, and $c \in B^{-1 + \gamma_c / 2}_{\max\{2+\varepsilon, d\}, 1}(\Omega)$, $\varepsilon > 0$ then operator
$$
\mathcal{R}: H^{-1 + t }(\Omega) \to \tilde{H}^{1 + t}(\Omega), \,\,\,\,\, t\in [0,\gamma_c / 2)
$$
is bounded.
\end{corollary}

\medskip

\medskip

\medskip

\textbf{Aknowledgements.}

\medskip

The author takes the opportunity to express deep gratitude to O.V. Besov, V.I. Burenkov, M.L. Goldman, A.M. Stepin and A.I. Tulenev for useful discussions, links, comments and advices.

\end{document}